\documentclass{amsart}
\usepackage[english]{babel}
\usepackage[utf8]{inputenc}
\usepackage{amsmath, amssymb,epic,graphicx,mathrsfs,enumerate}
\usepackage[all]{xy}
\usepackage{color}
\usepackage{comment}
\usepackage{enumitem}
\usepackage{hyperref}
\usepackage[]{frontespizio}

\DeclareMathOperator{\aut}{Aut}

 \DeclareMathOperator{\frat}{Frat}

\DeclareMathOperator{\GL}{GL}

\DeclareMathOperator{\Ee}{E}

\DeclareMathOperator{\alt}{Alt}
\DeclareMathOperator{\End}{End} 
 
\DeclareMathOperator{\aff}{Aff}
\DeclareMathOperator{\fit}{Fit}

\newcommand{\FF}{\mathbb F}

\renewcommand{\emptyset}{\varnothing}

\newtheorem{thm}{Theorem}
\newtheorem{cor}[thm]{Corollary}

 \newtheorem{lemma}[thm]{Lemma}
\newtheorem{prop}[thm]{Proposition}

\numberwithin{equation}{section}

\renewcommand{\footnote}{\endnote}
\newcommand{\ignore}[1]{}\makeglossary
\begin{document}
\bibliographystyle{amsplain}
\subjclass{20P05}
\keywords{groups generation; waiting time; profinite groups}
\title[Strongly generating elements]{Strongly generating elements\\ in finite and profinite groups}
\author{Eloisa Detomi}
\address{Dipartimento di Ingegneria dell'Informazione - DEI, Universit\`a 
	di Padova, Via G. Gradenigo 6/B, 35121 Padova, Italy} 
\email{eloisa.detomi@unipd.it}
\author{Andrea Lucchini}
\address{Dipartimento di Matematica \lq\lq Tullio Levi Civita\rq\rq,  Universit\`a 
	di Padova, Via Trieste 63, 35121 Padova, Italy}\email{lucchini@math.unipd.it}

\begin{abstract}Given a finite group $G$ and an element $g\in G$, we may compare the expected number $e(G)$ of elements needed to generate $G$ and the expected number $e(G,g)$ of elements of $G$ needed to generate $G$ together with $g.$ We address the following question: how large can the difference $e(G)-e(G,g)$ be?
\end{abstract}
\maketitle
\section{introduction}
Let $G$ be a nontrivial finite group and let $x=(x_n)_{n\in\mathbb N}$ be a sequence of independent, uniformly distributed $G$-valued random variables.
Let $Y$ be a subset of $G.$
We may define a random variable $\tau_{G,Y}$ (a waiting time) by
$$\tau_{G,Y}=\min \{n \geq 1 \mid \langle Y, x_1,\dots,x_n \rangle = G\} \in [0,+\infty].$$
Since $\tau_{G,Y}>n$ if and only if $\langle Y, x_1,\dots,x_n \rangle \neq G$,  $P(\tau_{G,Y}>n)=1-P_{G,Y}(n),$ denoting by $$P_{G,Y}(n) =
\frac{|\{(g_1,\dots,g_n)\in G^n \mid \langle g_1,\dots,g_n, Y\rangle= G\}|}{|G|^n}$$ the probability that $n$ randomly chosen
elements of $G$ generate $G$ together with the elements of $Y.$
We denote by $e(G,Y)$ the expectation $\Ee(\tau_{G,Y})$ of this random variable.
In other word $e(G,Y)$ is the expected number of elements of $G$ which have to be drawn at random, with replacement,
before a set is found with the property that its union with $Y$ generates $G.$
When $Y=\emptyset,$ we just write $e(G)$ and $P_G(n)$ instead of $e(G,Y)$ and $P_{G,Y}(n).$

It follows from \cite[Theorem 11]{mon} that $e(G,Y)$ can be directly determined
using the M\"obius function defined on the
subgroup lattice  of $G$ by setting
$\mu_G(G)=1$ and $\mu _G(H)=-\sum _{H<K}\mu _G(K)$ for any $H<G$. Indeed the following formula holds:
$$	e(G,Y)=-\sum_{Y\leq H < G}\frac{\mu_G(H)|G|}{|G|-|H|}.
$$

Consider the particular case when $Y=\{g\}$ is a singleton. In this case we just write $e(G,g)$. Clearly $e(G,g)\leq e(G)$ (with equality if and only if $g$ belongs to the Frattini subgroup of $G$). The question that we want to address in this paper is whether there are elements $g$ with the property that $e(G,g)$ is \lq\lq much smaller\rq\rq \ of $e(G)$. We call these element \lq\lq strongly generating elements\rq\rq. 

To discuss this problem, it is useful to highlight the relation between $e(G,Y)$ and another invariant related with the enumeration of the maximal subgroups of $G.$
For $n\in \mathbb N,$ denote by $m_n(G,Y)$ the number of maximal subgroups of $G$ with index $n$ and containing $Y$. Let
$$\mathcal M(G,Y)=\sup_{n\geq 2}\frac{\log (m_n(G,Y))}{\log (n)}.$$
Actually  ${\mathcal M} (G)={\mathcal M} (G,\emptyset)$ is the \lq\lq polynomial degree\rq\rq \ of the rate of growth of $m_n(G,\emptyset)$. This rate has been studied for finite and profinite groups by
 several authors, including
 Mann, Shalev, Borovik, Jaikin-Zapirain and  Pyber 
(see \cite{bps}, \cite{jp}, \cite{m}, \cite{ms}).

In \cite[Theorem 1]{mosca}, improving an argument used by Lubotzky \cite{lub}, it is proved that $\lceil \mathcal M(G)\rceil - 4\leq e(G) \leq \lceil \mathcal M(G)\rceil+3$  for any finite group $G$. With precisely the same argument the following more general statement can be proved:

\begin{thm}\label{main}
Let $G$ be a finite group and $Y$ a subset of $G$. Then 
 $$\lceil \mathcal M(G,Y)\rceil - 4\leq e(G,Y) \leq \lceil \mathcal M(G,Y)\rceil+3.$$
\end{thm}

If follows immediately:

\begin{cor}\label{bound}
	Let $G$ be a finite group. If $g\in G,$ then
	$$e(G)-e(G,g)\leq \mathcal M(G)-\mathcal M(G,g)+8.$$
\end{cor}

Our first application of Corollary \ref{bound} is to show that the difference
$e(G)-e(G,g)$ can be arbitrarily large. More precisely we prove:

\begin{thm}\label{uno}For any positive integer $n,$ there exists a finite 2-generated group $G$ such that
	\begin{enumerate}
		\item $e(G)\geq n.$
		\item  $e(G,g)\leq 5$ for some $g\in G.$
		\end{enumerate}
\end{thm}

The definition of $e(G,Y)$ can be extended to the case of a (topologically) finitely
generated profinite group $G.$ 
If we denote with $\mu$ the
normalized Haar measure on $G$, so that $\mu(G)=1$, the
probability that $k$ random elements generate (topologically) $G$ together with the elements of a given subset $Y$ of $G$ is defined as
$$P_{G,Y}(k)=\mu(\{ (x_1,\ldots,x_k)\in G^k \mid \langle x_1,\ldots,x_k, Y\rangle
=G\}),$$ where $\mu$ denotes also the product measure on $G^k$. As it is proved for example in \cite[Section 6]{mon}, 
$e(G)=\sup_{N\in \mathcal N}e(G/N),$ where $\mathcal N$ is the set of the open normal subgroups of $G$. In the same way it can be easily proved that
\begin{equation*}e(G,Y)=\sup_{N\in \mathcal N}e(G/N,YN/N).
\end{equation*}
This implies that the inequalities in Theorem \ref{main} still hold if $G$ is a finitely generated profinite group. Recall that a profinite group $G$ is said to be positively finitely generated, PFG for
short, if $P_G(k)$ is positive for some natural number $k$ and 
that $G$ is said to have polynomial maximal subgroup growth  if $m_n(G) \leq \alpha n^\sigma$ for all $n$ (and for some constants $\alpha$ and $\sigma$). A celebrated result of Mann and Shalev \cite{ms} says that a profinite group is PFG if and only if it has polynomial maximal subgroup growth. Moreover, by Theorem \ref{main}, $G$ is PFG if and only if $e(G)< \infty.$ A strongly generating element in a profinite group $G$ which is not PFG can be defined as an element $g$ with the property that $e(G,g)$ is finite. As an immediate corollary of Theorem \ref{uno} and its proof we  will 
 deduce the following result.

\begin{cor}\label{coro}There exists a 2-generated profinite group $G$ such that $e(G)=\infty,$ but
	$e(G,g)\leq 5$ for some $g\in G.$
\end{cor}

An analogous of Theorem \ref{uno} cannot hold if we only consider finite soluble groups. Indeed denote by $d(G)$ the smallest cardinality of a generating set of $G.$ By \cite[Corollary 14]{mon}, there exists an absolute constant $\alpha$ such that, for any finite soluble group $G$, $e(G)\leq   \left\lceil \gamma d(G)  \right\rceil+\alpha,$ where
$\gamma \backsimeq 3.243$ is  the P\`{a}lfy-Wolf constant. Clearly, for any $g\in G,$ we have $e(G,g)\geq d(G)-1,$ so
 $e(G)\leq \gamma(e(G,g))+\gamma+\alpha+1$ and
\begin{equation*}
        e(G)-e(G,g)\leq \left\lceil \gamma d(G)  \right\rceil+\alpha - (d(G)-1)\leq (\gamma-1)d(G)+\alpha+2.
	\end{equation*}
In particular, if $e(G,g)$ is bounded, then $e(G)$ cannot be arbitrarily large. This does not exclude that a finite soluble group can contain strongly generating elements. Indeed the difference 
 $e(G)-e(G,g)$ could  increase  in function of $d(G).$ It can be immediately seen that this cannot be the case  if we consider finite nilpotent groups. Indeed if $G$ is a finite nilpotent group, then $e(G)\leq d(G)+2.1185$ (see for example \cite[Corollary 2]{pome}), and therefore $e(G)-e(G,g)\leq 3.1185.$ The same argument cannot be applied in the class of finite metabelian groups. Indeed it follows from \cite[Corollary 15]{mon} that for any $d\geq 2$ it is possible to construct a $d$-generated metabelian group $G$ with $e(G)\sim 2d.$ Nevertheless, a more careful computation shows that even in the class of finite metabelian group, we cannot have strongly generating elements. More precisely we prove the following more general statement. 
 \begin{thm}\label{meta}If the derived subgroup of a finite group $G$ is nilpotent, then 
 	$$e(G)-e(G,g)\leq 11$$ for every $g \in G.$
 \end{thm}

The previous theorem could cast doubt on the existence of strongly generating elements in finite soluble groups, but in the last section of the paper we prove that they  can occur. Indeed we prove:

\begin{thm}\label{strongsol}
For every positive integer $d\geq 2,$ there exists a finite soluble group $G$ with $d(G)=d$ and containing an element $g$ such that
$e(G)-e(G,g) \sim 0.262 \cdot d-8.$
\end{thm}

\section{Proofs of Theorem \ref{uno} and Corollary \ref{coro}}

Let $\mathcal P$ be the set of the primes $p\geq 7$ with the property that
$p\notin \{11,23\}$ and there is no prime power $q$ and positive integer $d$ with $p=(q^d-1)/(q-1).$ It is a formidable open problem in Number Theory whether  there are infinitely many primes $p$ which do not belong to $\mathcal P.$ However it can be easily seen that the set of these primes have density 0. This implies in particular that $\mathcal P$ is infinite.

Let $S=\alt(p),$ with $p\in \mathcal P,$ and $G_p=S^{(p-1)!}$. If follows from the results presented in \cite{mena} that $S^t$ is 2-generated if and only if $t\leq P_S(2)|S|^2/|\aut S|.$ In particular if $S=\alt(p),$ then $P_S(2)\geq P_{\alt(7)}(2)=\frac{229}{315},$ so $(p-1)!\leq \frac{229}{315}\frac{p!}{4}$
and therefore $S^{(p-1)!}$ is 2-generated.
Notice that $G_p$ contains $(p-1)!$ normal subgroups $S_1,\dots,S_{(p-1)!}$ with $S_i\cong S.$ For each $1\leq t\leq (p-1)!$ there is an epimorphism $\pi_t: G_p \to S$ with $\ker(\pi_t)=R_t,$ being $R_t=\prod_{j\neq t}S_j.$ Since $S$ contains $p$ maximal subgroups $M_1,\dots,M_p$ of index $p,$ $G_p$ contains $p\cdot (p-1)!$ maximal subgroups of index $p$ (namely $\pi_t^{-1}(M_u)$ for $1\leq t\leq (p-1)!$ and  $1\leq u\leq p$) so
$$e(G_p)\geq \mathcal M(G_p) - 4\geq \frac{\log m_p(G_p)}{\log p}-4=\frac{\log \left(p\cdot (p-1)!\right)}{\log p}-4\geq p\left(1-\frac{\log e}{\log p}\right)-4.$$

Now we set $g_p=(\sigma,\dots,\sigma),$ where $\sigma=(1,2,\dots,p)\in \alt(p)$, and we want to estimate $e(G_p,g_p).$ First recall that a maximal subgroup $M$ of $G$ can be of two possible kinds:
\begin{enumerate}
	\item There exists a maximal subgroup $H$ of $S$ and $1\leq t \leq (p-1)!$ such that $M=\pi_t^{-1}(H).$
	\item There exist $1\leq r < s\leq (p-1)!$ and $\alpha\in \aut(S)$ such that
	$M=\{x \in G_p \mid x^{\pi_s}=(x^{\pi_r})^\alpha\}.$
\end{enumerate}
If $p \in \mathcal P,$ then $N_S(\langle \sigma \rangle)\cong \aff(1,p)$ is the unique maximal subgroup of $S_p$ containing $\sigma$ (see for example \cite[Corollary 3]{ppp}), so there are precisely $(p-1)!$ maximal subgroups $M$ of  type (1) containing $g_p$, and $|G_p:M|=|S:\aff(1,p)|=(p-2)!$ for each of them. A maximal subgroup of the second type which contains $g_p$ is of the form $\{x \in G_p \mid x^{\pi_s}=(x^{\pi_r})^\alpha\}$ where $\sigma=g_p^{\pi_s}=(g_p^{\pi_r})^\alpha=\sigma^\alpha$ and has index $|S|.$ We have $
\binom{(p-1)!}{2}$ choices for the pair $(r,s)$ and $|C_S(\sigma)|=p$ choices for $\alpha.$ Summarizing $m_n(G_p,g_p)=0$ if $n\notin \{(p-2)!, p!/2\}$ and
$$
\begin{aligned}
\frac{\log \left(m_{(p-2)!}(G_p,g_p)\right)}{\log \left((p-2)! \right)}&=\frac{\log \left((p-1)! \right)}{\log \left((p-2)! \right)}\leq 2 \\
\frac{\log \left(m_{p!/2}(G_p,g_p)\right)}{\log \left(\frac{p!}2 \right)}&=
\frac
{\log\left(\frac{p\cdot (p-1)! \cdot \left((p-1)!-1 \right)}{2}\right)}
{\log\left(\frac{p!}{2}\right)}\leq 2
.
\end{aligned}$$
In particular $$e(G_p,g_p) \leq \lceil \mathcal M(G_p,g_p)\rceil+3\leq 5$$
and this concludes the proof of Theorem \ref{main}.

Now, let  $G_p$ and $g_p$ as defined above and let $G=\prod_{p\in \mathcal P}G_p$ and $g=(g_p)_{p\in \mathcal P}.$ For any $p \in \mathcal P,$ let $\pi_p: G\to G_p$ be the corresponding projection. 
We have $e(G)=\sup_{p\in \mathcal P}e(G_p)=\infty.$ Let $M$ be a maximal subgroup of $G$ containing $g.$ By the Goursat's Lemma (see for example \cite[4.3.1]{scott}), there exists $p \in \mathcal P$ and a maximal subgroup $L$ of $G_p$ containing $g_p$, such that $M=L \times \prod_{q\neq p}G_q.$ In particular $|G:M|=|G_p:L|\in \{(p-2)!,p!/2\}.$ Notice that if $p_1\neq p_2$ then
$\{(p_1-2)!,p_1!/2\}\cap \{(p_2-2)!,p_2!/2\}=\emptyset.$ If follows
$\mathcal M(G,g)=\sup_{p\in \mathcal P}\mathcal M(G_p,g_p)\leq 2,$ so $e(G,g)\leq 5.$ This proves Corollary \ref{coro}.

\section{Soluble groups}\label{solo}

Let $G$ be a finite soluble group and $M$ a maximal subgroup of $G$. Denote by $Y_M=\bigcap_{g\in G}M^g$ the normal core of $M$ in $G$ and by $X_M/Y_M$ the socle of the primitive permutation group $G/Y_M$: clearly $X_M/Y_M$ is a chief factor of $G$ and $M/Y_M$ is a complement of $X_M/Y_M$ in $G/Y_M.$ Let $\text {Max}(G)$ be the set of  maximal subgroups of $G,$ let $\mathcal V(G)$ be a set of representatives of the irreducible $G$-modules that are $G$-isomorphic to some complemented chief factor of $G$ and,
for every $V\in \mathcal V(G),$  let $\text{Max}(G,V)$ be the set of  maximal subgroups $M$ of $G$ with $X_M/Y_M\cong_G V.$

Recall some results by Gasch\"utz \cite{gaz}. Given $V\in\mathcal V(G)$, let $$R_G(V)  =\bigcap_{M\in\text{Max}(G,V)}M.$$
It turns out that $R_G(V)$ is the smallest normal subgroup contained
in $C_G(V)$ with the property that $C_G(V)/R_G(V)$ is
$G$-isomorphic to a direct product of copies of $V$ and that it has a
complement in $G/R_G(V)$. The factor group $C_G(V)/R_G(V)$ is
called the $V$-crown of $G$. The non-negative integer
$\delta_G(V)$ defined by $C_G(V)/R_G(V)$ $\cong_G V^{\delta_G(V)}$ is
called the $V$-rank of $G$ and it coincides with the number of
complemented factors in any chief series of $G$ that are
$G$-isomorphic to $V$ (see for example \cite[Section 1.3]{classes}). In particular $G/R_G(V)\cong V^{\delta_G(V)}\rtimes H,$ with $H\cong G/C_G(V).$
Let $\mathbb F_V=\End_G(V)$, $q_G(V)=|\mathbb F_V|$ and $r_G(V)=\dim_{\mathbb F_V}V.$ The maximal subgroups in $\text{Max}(G,V)$ are in bijective correspondence with the maximal supplements of $V^{\delta_G(V)}$ in $V^{\delta_G(V)}\rtimes H,$ i.e. with the subgroups of
$V^{\delta_G(V)}\rtimes H$ of the form $W H^w$ with $W$ a maximal $H$-submodule of $V^{\delta_G(V)}$ and $w\in V^{\delta_G(V)}$. The number of choices for $W$ are $(q_G(V)^{\delta_G(V)}-1)/(q_G(V)-1)$ (see for example \cite[Lemma 2.5]{dvl}) so
\begin{equation}\label{qui}|\text{Max}(G,V)|=\left(\frac{q_G(V)^{\delta_G(V)}-1}{q_G(V)-1}\right)|V|^{\theta_G(V)},
	\end{equation}
with $\theta_G(V)=0$ if $G$ centralizes $V,$ $\theta_G(V)=1$ otherwise. 
Now we want to count the number of maximal subgroups $M \in \text{Max}(G,V)$ containing a given element $g.$ Let us denote by $\text{Max}(G,g,V)$ this set. This task corresponds to count the number of maximal supplements of $V^{\delta_G(V)}$ in $V^{\delta_G(V)}\rtimes H,$ containing a given element $x.$ Suppose $x=wh,$ with $w\in V^{\delta_G(V)}$ and $h\in H.$ Then a maximal subgroup of $V^{\delta_G(V)}\rtimes H$ containing $x$ is of the form $WH^v$ with $v \in V^{\delta_G(V)},$ $W$ a maximal $H$-submodule of $V^{\delta_G(V)}$ and $w+[h^{-1},v]\in W.$ The multiplicity of $V$ in
$\mathbb F_VH/J(\mathbb F_VH)$ (where $\mathbb F_VH$ denote the group algebra and $J(\mathbb F_VH)$ its Jacobson radical) is
$r_G(V)$ so it follows from \cite[Lemma 1]{cgk} that the $H$-submodule of $V^{\delta_G(V)}$ generated by
one of its elements is isomorphic to $V^k$ with $k\leq r_G(A).$
So, if $\delta_G(V) > r_G(V),$ then for any choice of $v,$ we have at least
$(q_G(V)^{\delta_G(V)-r_G(V)}-1)(q_G(V)-1)$ maximal $H$-submodules $W$ containing $w+[h^{-1},v],$ and therefore
\begin{equation}\label{quo}|\text{Max}(G,g,V)|\geq \left(\frac{q_G(V)^{\delta_G(V)-r_G(V)}-1}{q_G(V)-1}\right)|V|^{\theta_G(V)}.
\end{equation}
Combining (\ref{qui}) and (\ref{quo}), it follows that if $\delta_G(V)>r_G(A),$ then
\begin{equation}\label{cippo}
|\text{Max}(G,V)|\leq |\text{Max}(G,g,V)|\cdot q_G(V)^{2r_G(V)}= |\text{Max}(G,g,V)|\cdot n^2.
\end{equation}
Let $\mathcal V_n(G):=\{V \in \mathcal V(G) \mid |V|=n\}$ and $\mathcal V_n^\circ(G)=\{ V \in \mathcal V_n(G) \mid \delta_G(V)\leq r_G(V)\}.$ Moreover
define
 $$\mu^+_n(G)=\sum_{V\in \mathcal V_n(G) \setminus \mathcal V_n^\circ(G)}|\text{Max}(G,V)|, \quad \mu_n^\circ(G)=\sum_{V\in \mathcal V_n^\circ(G)}|\text{Max}(G,V)|.$$
Notice that, by (\ref{cippo}), we have:
\begin{equation}\label{piumeno}
\mu^+_n(G) \leq m_n(G,g)n^2.
\end{equation}

\begin{lemma}\label{amelia}
	If $\mu^+_n(G)\geq  \mu_n^\circ(G)$, then, for every $g\in G,$
	$$\frac{\log m_n(G,g)}{\log n}\geq \frac{\log m_n(G)}{\log n}-3.
	$$
\end{lemma}
\begin{proof}Assuming $\mu^+_n(G)\geq  \mu_n^\circ(G)$, and using (\ref{piumeno}), we get
$$\begin{aligned}
\frac{\log m_n(G)}{\log n} &= \frac{\log \left(\sum_{V\in \mathcal V_n(G)}|\text{Max}(G,V)|\right)}{\log n} = \frac{\log\left(\mu^+_n(G)+ \mu_n^\circ(G)\right)}{\log n}\\
&\leq \frac{\log\left(2\mu^+_n(G)\right)}{\log n} \leq 1+\frac{\log\left(\mu^+_n(G)\right)}{\log n}\leq 1+\frac{\log\left(m_n(G,g)\cdot n^2\right)}{\log n}\\
&\leq 3+\frac{\log\left(m_n(G,g)\right)}{\log n}.\qedhere
\end{aligned}$$
\end{proof}
We say that $n$ is a $G$-dominant integer if 
$\mathcal M(G)=\lceil \log (m_n(G))/\log n\rceil.$ Combining Corollary \ref{bound} and Lemma \ref{amelia}, we deduce:

\begin{cor}\label{domino}Let $G$ be a finite soluble group. If $\mu^+_n(G)\geq  \mu_n^\circ(G)$ for some $G$-dominant integer $n,$ then $e(G)-e(G,g)\leq 11$ for every $g \in G.$
\end{cor}

Now we are ready to prove  Theorem \ref{meta}. 

\begin{proof}[Proof of Theorem \ref{meta}] 
 Suppose that the derived subgroup of $G$ is nilpotent.
Without loss of generality we can assume that $\frat(G)=1$. In this case the Fitting subgroup $\fit(G)$ of $G$ is a direct product of minimal normal subgroups of $G$, it is abelian and complemented. Let $K$ be a complement of $\fit(G)$ in $G$. Since $G^\prime$ is nilpotent, $G^\prime \leq \fit(G).$ In particular $K^\prime \leq \fit(G)\cap K=1,$ and therefore $K$ is abelian.
  Let $F$ be a complement of $Z(G)$ in $\fit(G)$
and let $H=Z(G)\times K.$ We have $G=F\rtimes H$ and
we can write $F$ in the form
$$F = V_1^{\delta_G(V_1)} \times \cdots \times V_r^{\delta_G(V_r)}$$
where $V_1, \ldots , V_r$ are irreducible nontrivial $H$-modules, pairwise not $H$-isomorphic. Let $\mathbb F_i=\End_H(V_i).$ Since  $H$  is  abelian, $r_G(V_i)=\dim _{\FF_i}V_i=1$ and $H/C_H(V_i)$ is isomorphic to a subgroup of $\FF_i^*.$	Suppose that $n$ is a $G$-dominant integer.
By Corollary \ref{domino}, we may assume $\mu^+_n(G) < \mu_n^\circ(G).$
Since $r_G(V)=1$ for every $V \in \mathcal V(G),$ $\delta_G(V)=1$ for every
$V \in \mathcal V_n^\circ(G).$
Moreover, if
$V \in \mathcal V_n(G),$ then either $V=V_i$ for some $1\leq i \leq r$ or
$V$ is a trivial-$G$-module (this occurs only if $n$ is a prime and it divides $|H|$). So
$$\mu_n^\circ(G)\leq 1+\tau,$$ where $\tau$ is the cardinality of the set $\mathcal T=\left\{i \in \{1,\dots,r\} \mid |V_i|=n \right\}.$
In particular,  to any $i\in\mathcal T,$
there corresponds a different nontrivial homomorphism from $H$ to $\FF_i^*\cong C_{n-1}$. It follows
$$ \tau\leq (n-1)^{d(H)}-1.$$ But then
$$\mu_n^\circ(G) \leq (n-1)^{d(H)}\leq (n-1)^{d(G)}.$$
By Theorem \ref{main}
$$e(G)	\leq \mathcal M(G) +4\leq \frac{\log(2\mu_n^\circ(G))}{\log n}+4 \leq \frac{\log \mu_n^\circ(G)}{\log n}+5\leq d(G)+5.$$
 Since
$e(G,g)\geq d(G)-1,$ we conclude that $e(G)-e(G,g)\leq 6.$
\end{proof}

\section{Proof of Theorem \ref{strongsol}}

Le $X$ be the Sylow 2-subgroup of $\GL(2,3).$ Recall that $X$ is the semidihedral group of order 16, it is 2-generated and $|\aut(X)|=16.$
Now let $F$ be the free group of rank $d\geq 2$ and denote by $\Sigma$ the set of the epimorphisms $\sigma: F \to X.$ We have
$$|\Sigma|=|X|^d P_X(d)\geq 16^d P_X(2)=\frac{3 \cdot 16^d}{8}.$$
Now let $$N:=\bigcap_{\sigma \in \Sigma}\ker \sigma \quad \text { and }\quad H:=F/N.$$
Then $H$ is a $d$-generated 2-group
 and it contains $\rho:=|\Sigma|/|\aut(X)|$ different normal subgroups $N_1,\dots,N_\rho$ with $H/N_i\cong X.$
If follows from \cite{pome} that
$$e(H)=e(C_2^d)=d+\sum_{1\leq i\leq d}\frac{1}{2^i-1}\leq d+\sum_{0\leq j\leq d-1}\frac{1}{2^j}\leq d+2.$$
We consider the semidirect product
$$G_d:=\prod_{1\leq i \leq \rho}V_i\rtimes H,$$
where $V_i\cong  C_3\times C_3$ and $N_i$ is the kernel of the action of $H$ on $V_i$. Notice that by the main result in \cite{eul}, $d(G_d)=d(H)=d.$
It follows that $$\mathcal M(G_d)=\frac{\log(m_9(G_d))}{\log(9)}\geq 
\frac{\log \left(\frac{3 \cdot 16^{d-1}}{8} \right)}{\log (9)}=
\frac{(d-1)\log(16)+\log(3)-\log(8)}{\log(9)}.$$
In particular
$$e(G_d)\geq \frac {d\log(16)}{\log(9)}-c,  \text { with }  c=4+\frac{\log (16)+\log (8)-\log (3)}{\log (9)} \leq 6.$$
Now consider $g=(v_1,\dots,v_\rho)\in \prod_{1\leq i \leq \rho}(V_i\setminus \{0\}).$ Then
$m_9(G_d,g)=0,$ so $e(G_d,g)\leq e(H)\leq d+2.$
We conclude that 
$$e(G_d)-e(G_d,g)\geq  \frac {d\log(16)}{\log (9)} - 6 -(d+2)=
d\left(\frac{\log(16)-\log(9)}{\log(9)}\right)-8\sim 0.262\cdot d-8.$$

\noindent {\bf Remark.} Notice that the previous argument can be repeated taking an arbitrary pair $(H,V),$ where $H$ is a nilpotent irreducible subgroup of $\GL(V)$ with $|H|>|V|.$ There exists a positive constant $c$, depending only on $H$, such that if $\tilde H_d$ is the largest $d$-generated subdirect product of copies of $H,$ and $\rho$ is the number of different kernels of epimorphisms from $\tilde H_d\to H,$ then the semidirect product $G_d=\prod_{1\leq i\leq \rho}V_i \rtimes \tilde H_d$ contains an element $g$ with the property that
$$e(G_d)-e(G_d,g)\geq  d\left(\frac {\log|H|}{\log|V|}-1\right)-c.$$
In the proof of \cite[Proposition 1.7]{wo}  a sequence
$(H_n,V_n)_{n\in \mathbb N}$ is constructed, where $H_n$ is a nilpotent irreducible subgroup of $\GL(V_n),$ such that
$$\lim_{n\to \infty}\frac{\log|H_n|}{\log |V_n|}=\frac{2^n\log (16)+(1+2+\dots+2^{n-1})\log (2)}{2^n\log (9)}=\frac{\log (32)}{\log (9)}
\sim 1.577,$$ hence, for $d$ large enough, it can be constructed a $d$-generated soluble group $G_d$
containing an element $g$ with $e(G_d)-e(G_d,g)\sim 0.577 \cdot d.$
By \cite[Theorem 1.6]{wo}, this is the largest  gap that can be obtained with these kind of examples.

\section{Other considerations about profinite groups}
As mentioned in the introduction, if $G$ is a $d$-generated prosoluble group, then $G$ is PFG and the gap $e(G)-e(G,g)$ is bounded, for every $g\in G,$ by 
 $e(\widehat F_{\text{sol},d}),$  
where $\widehat F_{\text{sol},d}$ is the free-prosoluble group of rank $d$ (i.e the prosoluble completion of the free group of rank $d$). 
 In the case where $G$ is precisely the free-prosoluble group, the arguments introduced in Section 3 allow us to show that this gap can be actually bounded by a constant that does not depend on  $d(G)$.

\begin{prop} There exists an absolute constant $\kappa $ 
such that 
	$$e(\widehat F_{\text{sol},d})-e(\widehat F_{\text{sol},d},g)\leq \kappa $$ 
	 	for every positive integer $d$ and for every $g \in \widehat F_{\text{sol},d}.$
\end{prop}
\begin{proof}A prosoluble  group is PFG \cite[Theorem 10]{m}, hence $e(\widehat F_{\text{sol},d}) < \infty$. More precisely there exists a constant $\alpha$ such that, for each $d \geq 2,$ we have
$e(\widehat F_{\text{sol},d}) \leq \lceil \gamma d\rceil + \alpha$, 
	where $\gamma \sim 3.243$ is the  P\`{a}lfy-Wolf  constant. Clearly it is not restrictive to assume $d>2.$ Let $G=\widehat F_{\text{sol},d}$ and
	let $K_n$ be the intersection of the maximal open subgroups of $G$ with index $n.$ Since $G$ is finitely generated, $K_n$ is an open normal subgroup of $G$,  $m_n(G)=m_n(G/K_n)$ and $m_n(G,g)=m_n(G/K_n,gK_n)$
	for every $g\in G.$ Let $V\in \mathcal V_n(G/K_n)$ and
set  $H=(G/K_n)/C_{G/K_n}(V) \cong G/C_G(V).$ It follows from the main result in \cite{eul} that the semidirect product $V^t \rtimes H$ is $d$-generated if and only if $t\leq r_{H}(V)(d-\theta_H(V)).$ Since $G$ is the free-prosoluble group of rank $d,$ every $d$-generated finite soluble group is an epimorphic image of $G$. It follows that $V^{r_H(V)(d-\theta_H(V))}\rtimes H$
	is an epimorphic image of $G/K_n$ and consequently $\delta_{G/K_n}(V)=r_H(V)(d-\theta_H(V)).$
	Since we are assuming $d>2,$ $\delta_{G/K_n}(V)>r_H(V)=r_{G/K_n}(V),$ i.e. $\mathcal V^\circ_n(G/K_n)=\emptyset.$ By (\ref{piumeno}) $m_n(G)=m_n(G/K_n)\leq m_n(G/K_n,gK_n)n^2 = m_n(G,g)n^2$ for every $g\in G$, and consequently, by Corollary \ref{bound},  $e(G)-e(G,g)\leq 10.$ 
\end{proof}

We conclude with some considerations on 
 strongly generating elements in the finitely generated profinite groups that are not PFG. Recall that when $G$ is a finitely generated group which is not PFG, then we say that $g\in G$ is a strongly generating element if $e(G,g)$ is finite. Denote by $\mathcal{S}(G)$ the set of these elements. Notice that $g\in \mathcal{S}(G)$ if and only if there exists a positive integer $k$ with the property that $P_{G,\{g\}}(k) < \infty.$
The set $\mathcal F(G)$ of the elements of $G$ which do not satisfy this property is studied in \cite{dL}. In particular,  by \cite[Proposition 2.6]{dL}, 
$\mu(\mathcal S(G))=1-\mu(\mathcal F(G))=0,$ in other words a non-PFG group $G$ can contain strongly generating elements, but the probability that a randomly chosen element of $G$ is a strongly generating element is zero.
Nevertheless, $\mathcal S(G)$ can be dense in $G$ (see \cite[Proposition 3.5]{dL}). On the other hand, there are examples of non-PFG groups without strongly generating elements. In particular, 
 $\mathcal S(\hat F_d)=\emptyset$ for every $d\geq 2$, denoting by $\hat F_d$ the free profinite group of rank $d$ \cite[Theorem 3.3]{dL}.

\end{document}